\documentclass[a4paper]{scrartcl}
\usepackage{amsmath}
\usepackage{amssymb}
\usepackage[T1]{fontenc}
\usepackage{amsthm}
\usepackage{longtable}
\usepackage{stmaryrd}
\usepackage{color}
\usepackage{hyperref}
\usepackage[ansinew]{inputenc}
\usepackage{graphicx}
\usepackage{tikz}

\begin{document}
\theoremstyle{remark}
\newtheorem{theorem}{\textbf{Theorem}}
\newtheorem{lemma}[theorem]{\textbf{Lemma}}
\newtheorem{corollary}[theorem]{\textbf{Corollary}}
\newtheorem{proposition}[theorem]{\textbf{Proposition}}
\newtheorem{claim}{\textbf{Claim}}
\newtheorem{algorithm}{\textbf{Algorithm}}
\newtheorem{definition}{\textbf{Definition}}
\newtheorem{case}{\textbf{Case}}
\newtheorem*{beweis}{\textbf{Proof:}}
\begin{center}
\large{\textsc{On the minimal monochromatic $K_4$-density}}
\end{center}
\begin{center}
\large{\textsc{Konrad Sperfeld}}\linebreak
Universität Rostock, Institut für Mathematik\linebreak
D-18057 Rostock, Germany \linebreak
\textit{Konrad.Sperfeld@uni-rostock.de}
\end{center}
\begin{abstract}
Abstract: We use Razborov's flag algebra method \cite{flagalgebra} to show a new asymptotic lower bound for the minimal density $m_4$ of monochromatic $K_4$'s in any $2$-coloring of the edges of the complete graph $K_n$ on $n$ vertices. The hitherto best known lower bound was obtained by Giraud \cite{giraud}, who proved that $m_4>\frac{1}{46}$, whereas the best known upper bound by Thomason \cite{thomason} states that $m_4<\frac{1}{33}$. We can show that $m_4>\frac{1}{35}$.
\end{abstract}
\section{Introduction}
Let $c_n$ be a $2$-coloring of the edges of the complete graph $K_n$ on $n$ vertices and $k_t(c_n)$ the number of monochromatic $K_t$'s in $c_n$. Now we denote by
\begin{equation}\label{eq:def}
m_t:=\lim\limits_{n \rightarrow \infty}{\underbrace{\frac{\min\left\{k_t(c_n)\right\}}{\binom{n}{t}}}_{=:r_{t}(n)}}
\end{equation}
the asymptotic value of the minimal density of monochromatic $K_t$'s in any $2$-coloring of a complete graph. Ramsey's theorem implies that $r_{t}(n)>0$ for large enough $n$ and since it is easily shown that $r_{t}(n)$ increases with $n$, it follows that the limit in equation \eqref{eq:def} exists. Thus, in every simple graph $G$ the minimal density of $K_t$'s in $G$ and its complement $\overline{G}$ is $m_t+o(1)$.\newline
As an easy consequence of a result of Goodman \cite{goodman} one gets $m_3=\frac{1}{4}$. But this is the only known value of $m_t$, apart from the trivial case $t=2$ with $m_2=1$. In 1964, Erdös \cite{erdoes} conjectured that $m_t=2^{1-\binom{t}{2}}$. This value can be achieved by a typical random graph. In 1989, Thomason \cite{thomason} disproved this conjecture by showing that $m_4<\frac{1}{33}$. But it seems that his construction is not optimal either. He uses a blowup operation to construct a graph with monochromatic $K_4$-density lower than $\frac{1}{33}$. But in the final paragraph of \cite{thomason2}, Thomason observes that it is possible to improve this construction by a tiny amount using a random perturbation. The only known lower bound for $m_4$ was given by Giraud \cite{giraud}. He proved that $m_4>\frac{1}{46}$. A shorter and English version of his proof can be found in \cite{wolf}. We will prove in section \ref{sec:main} that $m_4>\frac{1}{35}$.  For our proof we use Razborov's flag algebra method \cite{flagalgebra}. In section \ref{sec:flag} we will roughly explain everything we need from it to make this paper self contained.\newline
For a simple graph $G$ of order $n$, we write $V(G)$ for its set of vertices, $E(G)$ for its set of edges and $\overline{G}$ for its complement. Thus, $E(G)=E(K_n)\backslash E(\overline{G})$. Furthermore, let $\left[k\right]:=\left\{1,2,\ldots,k\right\}$. We write vectors underlined, e.g. $\underline{v}=\left(\underline{v}(1),\underline{v}(2),\underline{v}(3)\right)$ is a vector with three coordinates. A collection $V_1,\ldots,V_t$ of finite sets is a sunflower with center $C$ if $V_i\cap V_j=C$ for every two distinct $i,j\in\left[t\right]$.
\section{Flag Algebras}\label{sec:flag}
With his theory of flag algebras, Razborov developed a very strong tool for solving some classes of problems in extremal graph theory. For our proof, we will just need a small part of his method, which can be thought of as an application of the Cauchy-Schwarz inequality in the theory of simple graphs. For a detailed study of flag algebras we refer the reader to Razborov's original paper \cite{flagalgebra}. In this section we will just define the most important ingredients for our calculation. Furthermore, we will give a short introduction into flag algebras, which is very similar to Razborov's presentation in \cite{razborov}.\newline
Let $\mathcal{G}$ be the family of all unlabeled simple graphs considered up to isomorphism. By $\mathcal{G}_\ell$ we denote the set of all $G\in\mathcal{G}$ with order $\ell$. A type $\sigma$ of order $k$ is a labeled graph of order $k$. Thus, each vertex of a type can be uniquely identified by its label. Usually, we use the elements of $\left[k\right]$ as labels. In our proof we will deal with six types $\sigma_0,\sigma_1,\ldots,\sigma_5$ of order $4$. These are defined in Figure \ref{fig:types}. 

\begin{figure}[h]
\begin{tikzpicture}
\filldraw[fill=black] (1,1) circle (0.07cm);
\draw (0.7,1) node {$1$};
\filldraw[fill=black] (1,2) circle (0.07cm);
\draw (0.7,2) node {$3$};
\filldraw[fill=black] (2,1) circle (0.07cm);
\draw (2.3,1) node {$2$};
\filldraw[fill=black] (2,2) circle (0.07cm);
\draw (2.3,2) node {$4$};
\draw (1.5,0.5) node {$\sigma_0$};
\end{tikzpicture}
$\; \;$
\begin{tikzpicture}
\draw (1,1)--(2,1);
\filldraw[fill=black] (1,1) circle (0.07cm);
\draw (0.7,1) node {$1$};
\filldraw[fill=black] (1,2) circle (0.07cm);
\draw (0.7,2) node {$3$};
\filldraw[fill=black] (2,1) circle (0.07cm);
\draw (2.3,1) node {$2$};
\filldraw[fill=black] (2,2) circle (0.07cm);
\draw (2.3,2) node {$4$};
\draw (1.5,0.5) node {$\sigma_1$};
\end{tikzpicture}
$\; \;$
\begin{tikzpicture}
\draw (1,1)--(2,1);
\draw (1,2)--(2,2);
\filldraw[fill=black] (1,1) circle (0.07cm);
\draw (0.7,1) node {$1$};
\filldraw[fill=black] (1,2) circle (0.07cm);
\draw (0.7,2) node {$3$};
\filldraw[fill=black] (2,1) circle (0.07cm);
\draw (2.3,1) node {$2$};
\filldraw[fill=black] (2,2) circle (0.07cm);
\draw (2.3,2) node {$4$};
\draw (1.5,0.5) node {$\sigma_2$};
\end{tikzpicture}
$\; \;$
\begin{tikzpicture}
\draw (1,2)--(1,1)--(2,1);
\filldraw[fill=black] (1,1) circle (0.07cm);
\draw (0.7,1) node {$1$};
\filldraw[fill=black] (1,2) circle (0.07cm);
\draw (0.7,2) node {$3$};
\filldraw[fill=black] (2,1) circle (0.07cm);
\draw (2.3,1) node {$2$};
\filldraw[fill=black] (2,2) circle (0.07cm);
\draw (2.3,2) node {$4$};
\draw (1.5,0.5) node {$\sigma_3$};
\end{tikzpicture}
$\; \;$
\begin{tikzpicture}
\draw (1,2)--(1,1)--(2,1);
\draw (1,1)--(2,2);
\filldraw[fill=black] (1,1) circle (0.07cm);
\draw (0.7,1) node {$1$};
\filldraw[fill=black] (1,2) circle (0.07cm);
\draw (0.7,2) node {$3$};
\filldraw[fill=black] (2,1) circle (0.07cm);
\draw (2.3,1) node {$2$};
\filldraw[fill=black] (2,2) circle (0.07cm);
\draw (2.3,2) node {$4$};
\draw (1.5,0.5) node {$\sigma_4$};
\end{tikzpicture}
$\; \;$
\begin{tikzpicture}
\draw (1,2)--(1,1)--(2,1)--(2,2);
\filldraw[fill=black] (1,1) circle (0.07cm);
\draw (0.7,1) node {$1$};
\filldraw[fill=black] (1,2) circle (0.07cm);
\draw (0.7,2) node {$3$};
\filldraw[fill=black] (2,1) circle (0.07cm);
\draw (2.3,1) node {$2$};
\filldraw[fill=black] (2,2) circle (0.07cm);
\draw (2.3,2) node {$4$};
\draw (1.5,0.5) node {$\sigma_5$};
\end{tikzpicture}
\caption{Definition of the types $\sigma_0,\sigma_1,\sigma_2,\sigma_3,\sigma_4,\sigma_5$.}\label{fig:types}
\end{figure}
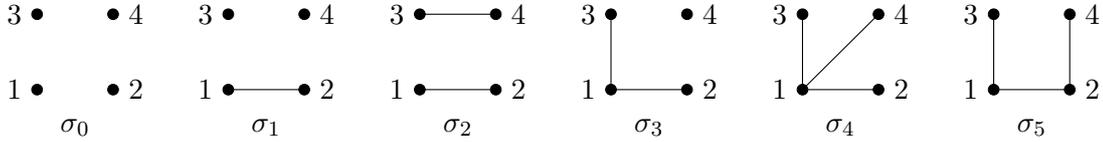
One denotes by $0$ the unique type of order $0$. Likewise one denotes by $1$ the unique type of order $1$.\newline
If $\sigma$ is a type of order $k$, we define a $\sigma$-flag as a pair $F=(G,\theta)$, where $G\in\mathcal{G}$ with $\left|V(G)\right|\geq k$ and $\theta:\left[k\right]\rightarrow V(G)$ is an injective function, such that the labelled vertices define an induced embedding of $\sigma$ into $G$. An isomorphism between two $\sigma$-flags $(G,\theta)$ and $(G',\theta')$ is an isomorphism $\phi$ between $G$ and $G'$ where $\phi(\theta(i))=\theta'(i)$. We write $\mathcal{F}^{\sigma}$ for the set of all $\sigma$-flags up to isomorphism. Again, we define $\mathcal{F}^\sigma_\ell\subseteq\mathcal{F}^\sigma$ as the set of all $\sigma$-flags of order $\ell$. For example, $\mathcal{F}^0_\ell=\mathcal{G}_\ell$.  If $\sigma$ is a type of order $k$, then $\mathcal{F}^\sigma_k$ consists only of $(\sigma,id)$. One denotes this element simply by $1_\sigma$. By definition we know that for every type $\sigma $ of order $k$ we have $\left|\mathcal{F}^{\sigma}_{k+1}\right|=2^k$. Thus following notation in \cite{wheel}, for $i=0,1,\ldots,5$ and $V\subseteq \left[4\right]$ we can denote the elements of $\mathcal{F}^{\sigma_i}_{5}$ by $F_V^{\sigma_i}=(G,\theta)$, such that $F_V^{\sigma_i}$ is the flag in which the only unlabelled vertex is connected to the set $\left\{\theta(i):\;i\in V\right\}$.\newline
Follow the notation of \cite{flagalgebra}, we write $\mathbf{math}$ $\mathbf{bold}$ $\mathbf{face}$ for random objects.
\begin{definition}(from \cite{flagalgebra})\newline
Fix a type $\sigma$ of order $k$, assume that integers $\ell,\ell_1,\ldots,\ell_t\geq k$ are such that
\begin{equation*}
\ell_1+\ldots+\ell_t-k(t-1)\leq \ell,
\end{equation*}
and $F=(M,\theta)\in \mathcal{F}_\ell^\sigma,\;F_1\in \mathcal{F}_{\ell_1}^\sigma,\ldots,\;\;F_t\in \mathcal{F}_{\ell_t}^\sigma$ are $\sigma$-flags. We define the (key) quantity $p(F_1,\ldots,F_t;F)\in\left[0,1\right]$ as follows. Choose in $V(M)$ uniformely at random a sunflower $(\mathbf{V_1},\ldots,\mathbf{V_t})$ with center $im(\theta)$ and $\forall i\;\left|\mathbf{V_i}\right|=\ell_i$. We let $p(F_1,\ldots,F_t;F)$ denote the probability of the event "$\forall i \in \left[t\right]\;F|_{\mathbf{V_i}}$ is isomorphic to $F_i$." When $t=1$, we use the notation $p(F_1,F)$ instead of $p(F_1;F)$.
\end{definition}
In the following we can identify a $\sigma$-flag $F$ by the probability $p(F,\hat{F})$, where $\hat{F}$ is an arbitrary large enough $\sigma$-flag. Thus, for example if we write
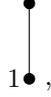
\begin{center}
\begin{tikzpicture}
\filldraw[fill=black] (1,0) circle (0.07cm);
\filldraw[fill=black] (1,1) circle (0.07cm);
\draw (1,0)--(1,1);
\draw (0.8,0) node {\small{$1$}};
\draw (1.2,-0.1) node {$\;$,};
\end{tikzpicture} 
\end{center}
we can think of it to be the normalized number of neighbours of a fixed vertex (called "$1$") in an arbitrary large enough graph. Or if we write
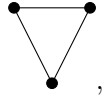
\begin{center}
\begin{tikzpicture}
\filldraw[fill=black] (1,1) circle (0.07cm);
\filldraw[fill=black] (0,1) circle (0.07cm);
\filldraw[fill=black] (0.5,0) circle (0.07cm);
\draw (0,1)--(0.5,0)--(1,1)--(0,1);
\end{tikzpicture},
\end{center}
we can think of it to be the triangle-density in an arbitrary large enough graph. Remark that these examples are not formal definitions. It should just allow an easier understanding of the following definitions. \newline 
Now, we build formal finite linear combinations of $\sigma$-flags. We denote the space which contains these linear combinations by $\mathbb{R}\mathcal{F}^\sigma$. Roughly speaking, if we think of the $F$-density in a graph of sufficently large order for a flag $F\in\mathcal{F}^\sigma_{\ell}$, it seems sensible to call the subspace $\mathcal{K}^\sigma$ which is generated by all elements of the form
\begin{equation*}\label{eq:dcargu}
F_1-\sum\limits_{\tilde{F}\in\mathcal{F}_{\tilde{\ell}}^\sigma}p(F_1,\tilde{F})\tilde{F},
\end{equation*}
where $F_1\in\mathcal{F}_{\ell_1}^\sigma$ with $\ell_1\leq\tilde{\ell}$, the subspace of "identically zero flag parameters". We want to illustrate this by an example. It can be seen by an easy double-counting argument that
\begin{equation}\label{eq:dcargu}
F_1=\sum\limits_{\tilde{F}\in\mathcal{F}_{\tilde{\ell}}^\sigma}p(F_1,\tilde{F})\tilde{F}.
\end{equation}
For example, the edge-density in an arbitrary large enough graph can be expressed as a linear combination of induced subgraph-densities of graphs of order $3$ in this graph. Thus,
\begin{equation*}
\parbox[c]{0.3cm}{
\begin{tikzpicture}
\filldraw[fill=black] (1,0) circle (0.07cm);
\filldraw[fill=black] (1,1) circle (0.07cm);

\draw (1,0)--(1,1);
\end{tikzpicture}}\;=\;
\parbox[c]{1cm}{
\begin{tikzpicture}
\filldraw[fill=black] (1,1) circle (0.07cm);
\filldraw[fill=black] (0,1) circle (0.07cm);
\filldraw[fill=black] (0.5,0) circle (0.07cm);
\draw (0,1)--(0.5,0)--(1,1)--(0,1);
\end{tikzpicture}}
\;+\;\frac{2}{3}
\parbox[c]{1cm}{
\begin{tikzpicture}
\filldraw[fill=black] (1,1) circle (0.07cm);
\filldraw[fill=black] (0,1) circle (0.07cm);
\filldraw[fill=black] (0.5,0) circle (0.07cm);
\draw (0,1)--(0.5,0)--(1,1);
\end{tikzpicture}}
\;+\;\frac{1}{3}
\parbox[c]{1cm}{
\begin{tikzpicture}
\filldraw[fill=black] (1,1) circle (0.07cm);
\filldraw[fill=black] (0,1) circle (0.07cm);
\filldraw[fill=black] (0.5,0) circle (0.07cm);
\draw (1,1)--(0,1);
\end{tikzpicture}}\;\;.
\end{equation*}
\newline
Now it is natural to define $\mathcal{A}^\sigma:=\mathbb{R}\mathcal{F}^\sigma / \mathcal{K}^\sigma $ as the flag algebra of the type $\sigma$. This means, we factor $\mathbb{R}\mathcal{F}^\sigma$ by the subspace $\mathcal{K}^\sigma$. In Lemma 2.4 of \cite{flagalgebra}, Razborov shows that $\mathcal{A}^\sigma$ is naturally endowed with the structure of a commutative associative algebra. He defines a bilinear mapping for flags in the following way. Let $\sigma$ be a type of order $k$. For two $\sigma$-flags $F_1\in\mathcal{F}_{\ell_1}^\sigma$, $F_2\in\mathcal{F}_{\ell_2}^\sigma$ and $\ell\geq \ell_1+\ell_2-k$ we define
\begin{equation*}
F_1\cdot F_2:=\sum\limits_{F\in\mathcal{F}_\ell^\sigma}p(F_1,F_2;F)F.
\end{equation*}
Remark that this definition is not well defined on $\mathbb{R}\mathcal{F}^\sigma$, but on $\mathcal{A}^\sigma$ it is. The disadvantage of this definition is that this product is just asymptotically the same as the product one would expect, if we interpret the $\sigma$-flags in the above way, because
\begin{equation*}
p(F_1,F_2;F)=p(F_1,F)p(F_2,F)+o(1).
\end{equation*}
That is why flagalgebraic proofs using this product operation are only asymptotically true.\newline 
Additionally, we want to remark in crude words that the function $F\rightarrow p(F,\hat{F})$ for very large $\hat{F}$ asymptotically corresponds to an algebra homomorphism $\phi \in \text{Hom}(\mathcal{A}^\sigma,\mathbb{R})$. 
Razborov now considers the set
\begin{equation*}
\text{Hom}^+(\mathcal{A}^\sigma,\mathbb{R}):=\left\{\phi \in \text{Hom}(\mathcal{A}^\sigma,\mathbb{R}) | \forall F \in \mathcal{F}^\sigma \; \phi (F)\geq 0\right\}
\end{equation*}
and shows in Corollary 3.4 of \cite{flagalgebra} that $\text{Hom}^+(\mathcal{A}^\sigma,\mathbb{R})$ captures all asymptotically true relations in extremal combinatorics.\newline
Thus, we have seen the basic idea of flag algebras. It is useful to define for $f,g\in \mathcal{A}^0$ that $f\geq g$ if $\forall \phi \in \text{Hom}^+(\mathcal{A}^\sigma,\mathbb{R}) \left(\phi(f)\geq\phi(g)\right)$. This is a partial preorder on $\mathcal{A}^0$. Now we want to turn our attention to an application of the Cauchy-Schwarz inequality in flag algebras.  \newline
We define the averaging operator $\left\llbracket \cdot \right\rrbracket_\sigma\;:\;\mathcal{A}^\sigma\rightarrow \mathcal{A}^0$ as follows. For a type $\sigma$ of order $k$ and $F=(G,\theta)\in\mathcal{F}^\sigma$, let $q_\sigma(F)$ be the probability that a uniformely at random chosen injective mapping $\mbox{\boldmath$\theta$}:\;\left[k\right]\rightarrow V(G)$ defines an induced embedding of $\sigma$ in $G$ and the resulting $\sigma$-flag $(G,\mbox{\boldmath$\theta$})$ is isomorphic to $F$. Now, we define 
\begin{equation*}
\left\llbracket F\right\rrbracket_\sigma:=q(F)\cdot G
\end{equation*}
partially on $\mathcal{F}^\sigma$. 
In section 2.2 in \cite{flagalgebra}, Razborov proves that this operator can be extended linearly to $\mathcal{A}^\sigma$ and he explains why it corresponds to averaging.
\begin{theorem}\label{th:cs}\textit{Cauchy-Schwarz inequality}(from \cite{flagalgebra}, Theorem 3.14)\newline
Let $f,g\in\mathcal{F}^\sigma$, then
\begin{equation*}
\left\llbracket f^2\right\rrbracket_\sigma \cdot\left\llbracket g^2\right\rrbracket_\sigma \geq \left\llbracket fg\right\rrbracket_\sigma^2.
\end{equation*}
In particular $(g=1_\sigma)$,
\begin{equation*}
\left\llbracket f^2\right\rrbracket_\sigma \cdot \sigma \geq \left\llbracket f\right\rrbracket_\sigma^2,
\end{equation*}
which in turn implies
\begin{equation*}
\left\llbracket f^2\right\rrbracket_\sigma\geq 0.
\end{equation*}
\end{theorem}
As an example, we want to show that $m_3=\frac{1}{4}$. \newline
At first, we consider the normalized number of pairs of neighbours of one fixed vertex. There are two options, either there is an edge or there is no edge between such a pair. Thus, we can express this normalized number by the following linear combination of $\sigma$-flags, which is asymptotically the same as the square of the normalized number of neighbours of the fixed vertex.
\begin{equation*}
\parbox[c]{1cm}{
\begin{tikzpicture}
\filldraw[fill=black] (1,1) circle (0.07cm);
\filldraw[fill=black] (0,1) circle (0.07cm);
\filldraw[fill=black] (0.5,0) circle (0.07cm);
\draw (0,1)--(0.5,0)--(1,1);
\draw (0,1)--(1,1);
\draw (0.3,0) node {\small{$1$}};
\end{tikzpicture}}\;+\;
\parbox[c]{1cm}{
\begin{tikzpicture}
\filldraw[fill=black] (1,1) circle (0.07cm);
\filldraw[fill=black] (0,1) circle (0.07cm);
\filldraw[fill=black] (0.5,0) circle (0.07cm);
\draw (0,1)--(0.5,0)--(1,1);
\draw (0.3,0) node {\small{$1$}};
\end{tikzpicture}}
\;=\; \left(
\parbox[c]{0.7cm}{
\begin{tikzpicture}
\filldraw[fill=black] (1,0) circle (0.07cm);
\filldraw[fill=black] (1,1) circle (0.07cm);
\draw (1,0)--(1,1);
\draw (0.8,0) node {\small{$1$}};
\end{tikzpicture}}
\right)^2
\end{equation*} 
Now we want to average this normalized number for all choices of fixed vertices. Then we get
\begin{equation}\label{eq:examplefi}
\left\llbracket
\left(
\parbox[c]{0.7cm}{
\begin{tikzpicture}
\filldraw[fill=black] (1,0) circle (0.07cm);
\filldraw[fill=black] (1,1) circle (0.07cm);
\draw (1,0)--(1,1);
\draw (0.8,0) node {\small{$1$}};
\end{tikzpicture}}
\right)^2\right\rrbracket_1 \;
 =\;\left\llbracket
 \parbox[c]{1cm}{
\begin{tikzpicture}
\filldraw[fill=black] (1,1) circle (0.07cm);
\filldraw[fill=black] (0,1) circle (0.07cm);
\filldraw[fill=black] (0.5,0) circle (0.07cm);
\draw (0,1)--(0.5,0)--(1,1);
\draw (0,1)--(1,1);
\draw (0.3,0) node {\small{$1$}};
\end{tikzpicture}}\;+\;
\parbox[c]{1.2cm}{
\begin{tikzpicture}
\filldraw[fill=black] (1,1) circle (0.07cm);
\filldraw[fill=black] (0,1) circle (0.07cm);
\filldraw[fill=black] (0.5,0) circle (0.07cm);
\draw (0,1)--(0.5,0)--(1,1);
\draw (0.3,0) node {\small{$1$}};
\end{tikzpicture}}\right\rrbracket_1
 \;=\;
\parbox[c]{1cm}{
\begin{tikzpicture}
\filldraw[fill=black] (1,1) circle (0.07cm);
\filldraw[fill=black] (0,1) circle (0.07cm);
\filldraw[fill=black] (0.5,0) circle (0.07cm);
\draw (0,1)--(0.5,0)--(1,1);
\draw (0,1)--(1,1);

\end{tikzpicture}}\;+\frac{1}{3}\;
\parbox[c]{1cm}{
\begin{tikzpicture}
\filldraw[fill=black] (1,1) circle (0.07cm);
\filldraw[fill=black] (0,1) circle (0.07cm);
\filldraw[fill=black] (0.5,0) circle (0.07cm);
\draw (0,1)--(0.5,0)--(1,1);
\end{tikzpicture}}\;\;.
\end{equation}
Analogously, one can derive 
\begin{equation}\label{eq:examplese}
\left\llbracket
\left(
\parbox[c]{0.7cm}{
\begin{tikzpicture}
\filldraw[fill=black] (1,0) circle (0.07cm);
\filldraw[fill=black] (1,1) circle (0.07cm);
\draw (0.8,0) node {\small{$1$}};
\end{tikzpicture}}
\right)^2\right\rrbracket_1 \;
 =\;\left\llbracket
 \parbox[c]{1cm}{
\begin{tikzpicture}
\filldraw[fill=black] (1,1) circle (0.07cm);
\filldraw[fill=black] (0,1) circle (0.07cm);
\filldraw[fill=black] (0.5,0) circle (0.07cm);
\draw (0.3,0) node {\small{$1$}};
\end{tikzpicture}}\;+\;
\parbox[c]{1.2cm}{
\begin{tikzpicture}
\filldraw[fill=black] (1,1) circle (0.07cm);
\filldraw[fill=black] (0,1) circle (0.07cm);
\filldraw[fill=black] (0.5,0) circle (0.07cm);
\draw (0,1)--(1,1);
\draw (0.3,0) node {\small{$1$}};
\end{tikzpicture}}\right\rrbracket_1
 \;=\;
\parbox[c]{1cm}{
\begin{tikzpicture}
\filldraw[fill=black] (1,1) circle (0.07cm);
\filldraw[fill=black] (0,1) circle (0.07cm);
\filldraw[fill=black] (0.5,0) circle (0.07cm);

\end{tikzpicture}}\;+\frac{1}{3}\;
\parbox[c]{1cm}{
\begin{tikzpicture}
\filldraw[fill=black] (1,1) circle (0.07cm);
\filldraw[fill=black] (0,1) circle (0.07cm);
\filldraw[fill=black] (0.5,0) circle (0.07cm);
\draw (0,1)--(1,1);
\end{tikzpicture}}\;\;.
\end{equation}
Using a similar double-counting idea as in equation \eqref{eq:dcargu}, one can see that
\begin{equation}\label{eq:3ordersum}
1=
\parbox[c]{1cm}{
\begin{tikzpicture}
\filldraw[fill=black] (1,1) circle (0.07cm);
\filldraw[fill=black] (0,1) circle (0.07cm);
\filldraw[fill=black] (0.5,0) circle (0.07cm);
\draw (0,1)--(0.5,0)--(1,1)--(0,1);
\end{tikzpicture}}
\;+\;
\parbox[c]{1cm}{
\begin{tikzpicture}
\filldraw[fill=black] (1,1) circle (0.07cm);
\filldraw[fill=black] (0,1) circle (0.07cm);
\filldraw[fill=black] (0.5,0) circle (0.07cm);
\draw (0,1)--(0.5,0)--(1,1);
\end{tikzpicture}}
\;+\;
\parbox[c]{1cm}{
\begin{tikzpicture}
\filldraw[fill=black] (1,1) circle (0.07cm);
\filldraw[fill=black] (0,1) circle (0.07cm);
\filldraw[fill=black] (0.5,0) circle (0.07cm);
\end{tikzpicture}}
\;+\;
\parbox[c]{1cm}{
\begin{tikzpicture}
\filldraw[fill=black] (1,1) circle (0.07cm);
\filldraw[fill=black] (0,1) circle (0.07cm);
\filldraw[fill=black] (0.5,0) circle (0.07cm);
\draw (1,1)--(0,1);
\end{tikzpicture}}\;\;.
\end{equation}
Now the equations \eqref{eq:examplefi}, \eqref{eq:examplese} and \eqref{eq:3ordersum} and an application of Theorem \ref{th:cs} tell us that
\begin{eqnarray*}
\parbox[c]{1cm}{
\begin{tikzpicture}
\filldraw[fill=black] (1,1) circle (0.07cm);
\filldraw[fill=black] (0,1) circle (0.07cm);
\filldraw[fill=black] (0.5,0) circle (0.07cm);
\draw (0,1)--(0.5,0)--(1,1)--(0,1);
\end{tikzpicture}}
\;+\;
\parbox[c]{1cm}{
\begin{tikzpicture}
\filldraw[fill=black] (1,1) circle (0.07cm);
\filldraw[fill=black] (0,1) circle (0.07cm);
\filldraw[fill=black] (0.5,0) circle (0.07cm);
\end{tikzpicture}}
&=&
\frac{3}{2}\left(\left\llbracket
\left(
\parbox[c]{0.7cm}{
\begin{tikzpicture}
\filldraw[fill=black] (1,0) circle (0.07cm);
\filldraw[fill=black] (1,1) circle (0.07cm);
\draw (1,0)--(1,1);
\draw (0.8,0) node {\small{$1$}};
\end{tikzpicture}}
\right)^2\right\rrbracket_1+\left\llbracket
\left(
\parbox[c]{0.7cm}{
\begin{tikzpicture}
\filldraw[fill=black] (1,0) circle (0.07cm);
\filldraw[fill=black] (1,1) circle (0.07cm);
\draw (0.8,0) node {\small{$1$}};
\end{tikzpicture}}
\right)^2\right\rrbracket_1\right)-\frac{1}{2} \\
& \geq & \frac{3}{2}\left(
\left(\left\llbracket\
\parbox[c]{0.6cm}{
\begin{tikzpicture}
\filldraw[fill=black] (1,0) circle (0.07cm);
\filldraw[fill=black] (1,1) circle (0.07cm);
\draw (1,0)--(1,1);
\draw (0.8,0) node {\small{$1$}};
\end{tikzpicture}}
\right\rrbracket_1 \right)^2 
+
\left(\left\llbracket\
\parbox[c]{0.6cm}{
\begin{tikzpicture}
\filldraw[fill=black] (1,0) circle (0.07cm);
\filldraw[fill=black] (1,1) circle (0.07cm);
\draw (0.8,0) node {\small{$1$}};
\end{tikzpicture}}
\right\rrbracket_1 \right)^2 
\right)-\frac{1}{2} \\
& = & \frac{3}{2}\left(
\left(\parbox[c]{0.2cm}{
\begin{tikzpicture}
\filldraw[fill=black] (1,0) circle (0.07cm);
\filldraw[fill=black] (1,1) circle (0.07cm);
\draw (1,0)--(1,1);
\end{tikzpicture}} \right)^2
+
\left(1-\parbox[c]{0.2cm}{
\begin{tikzpicture}
\filldraw[fill=black] (1,0) circle (0.07cm);
\filldraw[fill=black] (1,1) circle (0.07cm);
\draw (1,0)--(1,1);
\end{tikzpicture}} \right)^2
\right) \\
& = & 3
\left(\parbox[c]{0.2cm}{
\begin{tikzpicture}
\filldraw[fill=black] (1,0) circle (0.07cm);
\filldraw[fill=black] (1,1) circle (0.07cm);
\draw (1,0)--(1,1);
\end{tikzpicture}} \right)^2
-3\,
\parbox[c]{0.2cm}{
\begin{tikzpicture}
\filldraw[fill=black] (1,0) circle (0.07cm);
\filldraw[fill=black] (1,1) circle (0.07cm);
\draw (1,0)--(1,1);
\end{tikzpicture}} 
+1.
\end{eqnarray*}
The righthandside depends only on the edge-density \parbox[c]{0.2cm}{\begin{tikzpicture}[scale=0.35]
\filldraw[fill=black] (1,0) circle (0.07cm);
\filldraw[fill=black] (1,1) circle (0.07cm);
\draw (1,0)--(1,1);
\end{tikzpicture}}, which can be minimized by taking $\parbox[c]{0.2cm}{\begin{tikzpicture}[scale=0.35]
\filldraw[fill=black] (1,0) circle (0.07cm);
\filldraw[fill=black] (1,1) circle (0.07cm);
\draw (1,0)--(1,1);
\end{tikzpicture}}=\frac{1}{2}$. Thus, we have
\begin{equation*}
\parbox[c]{1cm}{
\begin{tikzpicture}
\filldraw[fill=black] (1,1) circle (0.07cm);
\filldraw[fill=black] (0,1) circle (0.07cm);
\filldraw[fill=black] (0.5,0) circle (0.07cm);
\draw (0,1)--(0.5,0)--(1,1)--(0,1);
\end{tikzpicture}}
\;+\;
\parbox[c]{1cm}{
\begin{tikzpicture}
\filldraw[fill=black] (1,1) circle (0.07cm);
\filldraw[fill=black] (0,1) circle (0.07cm);
\filldraw[fill=black] (0.5,0) circle (0.07cm);
\end{tikzpicture}}\;\;\geq \frac{1}{4}\;\Rightarrow\;m_3\geq \frac{1}{4}.
\end{equation*}
On the other hand it is easy to see that $m_3\leq\frac{1}{4}$, since a typical random graph achieves $\frac{1}{4}$.\newline
We need some further notation for our calculations. For a $\sigma$-flag $F=(G,\theta)$ we define its complement $\overline{F}\in\mathcal{F}^{\overline{\sigma}}$ as $\left(\overline{G},\theta\right)$. If $\underline{v}$ is a vector of $n$ $\sigma$-flags, then $\overline{\underline{v}}:=\left(\overline{\underline{v}(1)},\ldots,\overline{\underline{v}(n)}\right)$. Let $\sigma$ be a type, then we denote by $Aut(\sigma)$ its group of automorphisms (for example $Aut(\sigma_4)\cong S_3$ and $Aut(\sigma_1)\cong\mathbb{Z}_2\times\mathbb{Z}_2$). Furthermore, we define $f_V^\sigma \in \mathcal{A}^\sigma$ by
\begin{equation*}
f_V^\sigma:=\sum\limits_{\eta\in Aut(\sigma)}F_{\eta(V)}^\sigma.
\end{equation*}
Notice that these elements are $Aut(\sigma)$-invariant. It is an easy consequence of Theorem~\ref{th:cs} that $\left\llbracket\underline{v}^TA\underline{v}\right\rrbracket_\sigma\geq 0$, if $\underline{v}$ is a vector of $n$ $\sigma$-flags and $A\in\mathbb{R}^{n\times n}$ is symmetric positive semidefinite.\newline
Each element of $\mathcal{A}^\sigma$ can be written as a direct sum of its $Aut(\sigma)$-invariant part and its $Aut(\sigma)$-antiinvariant part ($f\in\mathcal{A}^\sigma$ is called antiinvariant, if $\sum\limits_{\eta \in Aut(\sigma)}\eta(f)=0$). It is easy to see that if $f\in \mathcal{A}^\sigma$ is invariant and $g\in \mathcal{A}^\sigma$ is antiinvariant, then $\left[\left[fg\right]\right]_\sigma=0$. Hence, if we work with positive semidefinite matrices $A$ like above, we can split them for each type in an invariant part (denoted by "$^+$") and an antiinvariant part (denoted by "$^-$"). For a more detailed explanation of this direct sum decomposition see e.g. section 4 of \cite{razborov}.
\section{Main Result}\label{sec:main}
A lot of calculations in our proof deal with linear combinations of elements of $\mathcal{G}_6$. We have written a computer program to determine the elements of $\mathcal{G}_6$. It turns out that $\left|\mathcal{G}_6\right|=156$. In Appendix \ref{app:m6} we give a list of the elements $G_i$ for $i\in\left\{0,\ldots,155\right\}$ of $\mathcal{G}_6$.
\begin{theorem}\label{th:main}
\begin{equation*}
m_4\geq\frac{1}{34.7858}
\end{equation*}
\end{theorem}
\begin{proof}
At first we define some vectors $g_i^*$ and matrices $A^{i*}$ (see Appendix \ref{app:vecmat}). With the help of these, the proof is given by the following inequality.\footnotesize
\begin{eqnarray}
\parbox[c]{1cm}{
\begin{tikzpicture}
\filldraw[fill=black] (1,1) circle (0.07cm);
\filldraw[fill=black] (0,1) circle (0.07cm);
\filldraw[fill=black] (1,0) circle (0.07cm);
\filldraw[fill=black] (0,0) circle (0.07cm);
\draw (0,0)--(0,1)--(1,0)--(1,1)--(0,1);
\draw (1,1)--(0,0)--(1,0);
\end{tikzpicture}}
\;+\;
\parbox[c]{1cm}{
\begin{tikzpicture}
\filldraw[fill=black] (1,1) circle (0.07cm);
\filldraw[fill=black] (0,1) circle (0.07cm);
\filldraw[fill=black] (1,0) circle (0.07cm);
\filldraw[fill=black] (0,0) circle (0.07cm);\end{tikzpicture}}
\;-\frac{1}{34.7858} & \geq &
\sum\limits_{i=0}^5\left\llbracket\left(\underline{g_i}^+\right)^TA^{i+}\underline{g_i}^+\right\rrbracket_{\sigma_i}+\sum\limits_{i\in\left\{1,2,3,5\right\}}\left\llbracket\left(\underline{g_i}^-\right)^TA^{i-}\underline{g_i}^-\right\rrbracket_{\sigma_i} \nonumber \\
 & & + \sum\limits_{i=0}^4\left\llbracket\left(\overline{\underline{g_i}^+}\right)^TA^{i+}\overline{\underline{g_i}^+}\right\rrbracket_{\sigma_i}+\sum\limits_{i=1}^3\left\llbracket\left(\overline{\underline{g_i}^-}\right)^TA^{i-}\overline{\underline{g_i}^-}\right\rrbracket_{\sigma_i} \label{eq:main} \\
 & \geq & 0 \nonumber
\end{eqnarray}
\normalsize
In Appendix \ref{app:calctable} one can find a table which helps to verify inequality \eqref{eq:main}.

\end{proof}
\subsection{Some remarks}
Most parts of the proof were done by a computer. At first we decided to work in $\mathcal{G}_6$. Then we took types of order $4$. Thus, if we take products of two $\sigma$-flags on $5$ vertices, where the types have order $4$, then our calculus works in $\mathcal{G}_6$. The vectors $g_i^+$ and $g_i^-$ are chosen in such a way that we have a generating system of the corresponding invariant or antiinvariant part of $\mathcal{A}^{\sigma_i}$. After that we used a computer program to calculate the equations which the semidefinite matrices $A_i^*$ have to fulfill such that we can prove theorem \ref{th:main}. Finally, the determination of the matrices was simply done by a sufficiently close rational approximation to the outcome of a numerical semidefinite-program-solver. The decision to ignore some antiinvariant parts and to work with the same matrices for the complementary types belong to computer experiments. \newline
The computational effort was low enough to try the proof in our setting (in $\mathcal{G}_6$ with types of order 4) in the most general way with 11 types and thus with $22$ matrices. But even then, we could not improve our lower bound of $m_4$ essentially (It was not possible to show $\frac{1}{34.7857}$).

\pagebreak
\begin{appendix}
\section{Appendix: The 156 Graphs}\label{app:m6}
The following table defines the $156$ graphs of $\mathcal{G}_6$.
\tiny
\begin{longtable}{c|l}
 & $E(G_i)$ \\
 \hline
$G_{0}$ & \\
 \hline
$G_{1}$ & $\left\{1,2\right\}$\\
 \hline
$G_{2}$ & $\left\{1,2\right\}$$\left\{1,3\right\}$\\
 \hline
$G_{3}$ & $\left\{1,2\right\}$$\left\{1,3\right\}$$\left\{2,3\right\}$\\
 \hline
$G_{4}$ & $\left\{1,2\right\}$$\left\{1,3\right\}$$\left\{2,3\right\}$$\left\{1,4\right\}$\\
 \hline
$G_{5}$ & $\left\{1,2\right\}$$\left\{1,3\right\}$$\left\{2,3\right\}$$\left\{1,4\right\}$$\left\{2,4\right\}$\\
 \hline
$G_{6}$ & $\left\{1,2\right\}$$\left\{1,3\right\}$$\left\{2,3\right\}$$\left\{1,4\right\}$$\left\{2,4\right\}$$\left\{3,4\right\}$\\
\end{longtable}
\normalsize
\section{Appendix: Definition of the vectors and matrices}\label{app:vecmat}
Here we define the vectors and matrices which are needed in the proof of Theorem \ref{th:main}. \newline
At first, we define the vectors which belong to the invariant parts of the corresponding types.
\begin{eqnarray*}
\underline{g_0}^+ & := & \left(f_\emptyset^{\sigma_0}, f_{\left\{1\right\}}^{\sigma_0}, f_{\left\{1,2\right\}}^{\sigma_0}, f_{\left\{1,2,3\right\}}^{\sigma_0}, f_{\left[4\right]}^{\sigma_0}\right) \\
\underline{g_1}^+ & := & \left(f_\emptyset^{\sigma_1}, f_{\left\{1\right\}}^{\sigma_1}, f_{\left\{3\right\}}^{\sigma_1}, f_{\left\{1,2\right\}}^{\sigma_1}, f_{\left\{1,3\right\}}^{\sigma_1}, f_{\left\{3,4\right\}}^{\sigma_1}, f_{\left\{1,2,3\right\}}^{\sigma_1}, f_{\left\{1,3,4\right\}}^{\sigma_1}, f_{\left[4\right]}^{\sigma_1}\right) \\
\underline{g_2}^+ & := & \left(f_\emptyset^{\sigma_2}, f_{\left\{1\right\}}^{\sigma_2}, f_{\left\{1,2\right\}}^{\sigma_2}, f_{\left\{1,3\right\}}^{\sigma_2}, f_{\left\{1,2,3\right\}}^{\sigma_2}, f_{\left[4\right]}^{\sigma_2}\right) \\
\underline{g_3}^+ & := & \left(f_\emptyset^{\sigma_3}, f_{\left\{1\right\}}^{\sigma_3}, f_{\left\{2\right\}}^{\sigma_3}, f_{\left\{4\right\}}^{\sigma_3}, f_{\left\{1,4\right\}}^{\sigma_3}, f_{\left\{2,3\right\}}^{\sigma_3}, f_{\left\{1,2\right\}}^{\sigma_3}, f_{\left\{2,4\right\}}^{\sigma_3}, f_{\left\{1,2,3\right\}}^{\sigma_3}, f_{\left\{1,2,4\right\}}^{\sigma_3}, f_{\left\{2,3,4\right\}}^{\sigma_3}, f_{\left[4\right]}^{\sigma_3}\right) \\
\underline{g_4}^+ & := & \left(f_\emptyset^{\sigma_4}, f_{\left\{1\right\}}^{\sigma_4}, f_{\left\{2\right\}}^{\sigma_4}, f_{\left\{1,2\right\}}^{\sigma_4}, f_{\left\{2,3\right\}}^{\sigma_4}, f_{\left\{1,2,3\right\}}^{\sigma_4}, f_{\left\{2,3,4\right\}}^{\sigma_4}, f_{\left[4\right]}^{\sigma_4}\right) \\
\underline{g_5}^+ & := & \left(f_\emptyset^{\sigma_5}, f_{\left\{1\right\}}^{\sigma_5}, f_{\left\{3\right\}}^{\sigma_5}, f_{\left\{1,2\right\}}^{\sigma_5}, f_{\left\{3,4\right\}}^{\sigma_5}, f_{\left\{1,3\right\}}^{\sigma_5}, f_{\left\{1,4\right\}}^{\sigma_5}, f_{\left\{1,2,3\right\}}^{\sigma_5}, f_{\left\{1,3,4\right\}}^{\sigma_5}, f_{\left[4\right]}^{\sigma_5}\right)
\end{eqnarray*}
Now we define the vectors which belong to the anti-invariant parts of the corresponding types.
\begin{eqnarray*}
\underline{g_1}^- & := & \left(F_{\left\{1\right\}}^{\sigma_1}-F_{\left\{2\right\}}^{\sigma_1}, F_{\left\{3\right\}}^{\sigma_1}-F_{\left\{4\right\}}^{\sigma_1}, F_{\left\{1,3\right\}}^{\sigma_1}-F_{\left\{2,3\right\}}^{\sigma_1}, F_{\left\{1,3\right\}}^{\sigma_1}-F_{\left\{1,4\right\}}^{\sigma_1}, F_{\left\{1,3\right\}}^{\sigma_1}-F_{\left\{2,4\right\}}^{\sigma_1},\right. \\
& & \left.F_{\left\{1,3,4\right\}}^{\sigma_1}-F_{\left\{2,3,4\right\}}^{\sigma_1}, F_{\left\{1,2,3\right\}}^{\sigma_1}-F_{\left\{1,2,4\right\}}^{\sigma_1}\right) \\
\underline{g_2}^- & := & \left(F_{\left\{1\right\}}^{\sigma_2}-F_{\left\{2\right\}}^{\sigma_2}, F_{\left\{1,2\right\}}^{\sigma_2}-F_{\left\{3,4\right\}}^{\sigma_2}, F_{\left\{1,3\right\}}^{\sigma_2}-F_{\left\{1,4\right\}}^{\sigma_2}, F_{\left\{1,2,3\right\}}^{\sigma_2}-F_{\left\{1,2,4\right\}}^{\sigma_2},\right. \\
& & \left.F_{\left\{1,2,3\right\}}^{\sigma_2}-F_{\left\{1,3,4\right\}}^{\sigma_2}, F_{\left\{1,3\right\}}^{\sigma_2}-F_{\left\{2,4\right\}}^{\sigma_2}\right) \\
\underline{g_3}^- & := & \left(F_{\left\{2\right\}}^{\sigma_3}-F_{\left\{3\right\}}^{\sigma_3}, F_{\left\{1,2\right\}}^{\sigma_3}-F_{\left\{1,3\right\}}^{\sigma_3}, F_{\left\{2,4\right\}}^{\sigma_3}-F_{\left\{3,4\right\}}^{\sigma_3}, F_{\left\{1,2,4\right\}}^{\sigma_3}-F_{\left\{1,3,4\right\}}^{\sigma_3},\right. \\
\underline{g_5}^- & := & \left(F_{\left\{1\right\}}^{\sigma_5}-F_{\left\{2\right\}}^{\sigma_5}, F_{\left\{3\right\}}^{\sigma_5}-F_{\left\{4\right\}}^{\sigma_5}, F_{\left\{1,3\right\}}^{\sigma_5}-F_{\left\{2,4\right\}}^{\sigma_5}, F_{\left\{1,4\right\}}^{\sigma_5}-F_{\left\{2,3\right\}}^{\sigma_5},\right. \\
& & \left.F_{\left\{1,2,3\right\}}^{\sigma_5}-F_{\left\{1,2,4\right\}}^{\sigma_5}, F_{\left\{1,3,4\right\}}^{\sigma_5}-F_{\left\{2,3,4\right\}}^{\sigma_5}\right)
\end{eqnarray*}
Furthermore, the definition of the required symmetric positive definite matrices.
\begin{equation*}
A^{0+}:=\frac{1}{10^{10}}
\left(
\begin{array}{*{5}{r}}
    16862005  &   5938009 &   10666588 &    3932432 &  -16602234 \\
     5938009  &   2092561 &    3755186 &    1382600 &   -5846206 \\
    10666588  &   3755186 &    6748387 &    2489266 &  -10502546 \\
     3932432  &   1382600 &    2489266 &     920571 &   -3872384 \\
   -16602234  &  -5846206 &  -10502546 &   -3872384 &   16346570
\end{array}
\right)
\end{equation*}
\footnotesize
\begin{eqnarray*}
A^{1+}& := & \frac{1}{10^{10}}
\left(
\begin{array}{*{6}{r}}
   4320915081 &  5114912033 & 0838876074 & 20166387   &-720812722  &\\
   5114912033 &  7305691770 & 1132075909 &-213140155  &-239779535  &\\
    838876074 &  1132075909 & 4258272084 &-4128294302 &-2560148834 &\\
     20166387 & -0213140155 &-4128294302 & 4176995622 & 2383723976 &\\
  - 720812722 & -0239779535 &-2560148834 & 2383723976 & 1953845159 & \ldots \\
  -3536992459 & -3834927363 & 1822789967 &-2569691611 &-736358264  &\\
    832941643 &  0635345271 & 0359072249 &-166507792  &-472411192  &\\
  -3812420339 & -5515309466 & 3014089726 &-3717449251 &-2224367556 &\\
  -1928295121 & -2126188382 &-2780770750 & 2399954498 & 1886952203 &
\end{array}
\right. \\
& & \left.
\begin{array}{*{5}{r}}
 &  -3536992459 &   832941643& -3812420339 &-1928295121\\
 &  -3834927363 &   635345271& -5515309466 &-2126188382\\
 &   1822789967 &   359072249&  3014089726 &-2780770750\\
 &  -2569691611 &  -166507792& -3717449251 & 2399954498\\
\ldots &  -736358264 &  -472411192& -2224367556 & 1886952203 \\
 &   4494029803 &  -646218304&  5173661199 & 160104823 \\
 &  -646218304 &   292587053& -217236155 &-567828688   \\
 &   5173661199 &  -217236155&  7837719380 &-728044396 \\
 &   160104823 &  -567828688& -728044396 & 2327049595
\end{array}\right)
\end{eqnarray*}
\begin{equation*}
A^{2+}:=\frac{1}{10^{10}}
\left(
\begin{array}{*{6}{r}}
   159078056 &  307070840 & -132946711 &   37583858 & -260766405 &  -15356594\\
   307070840 &  592875018 & -256638654 &   72586497 & -503453484 &  -29630203\\
  -132946711 & -256638654 &  111109118 &  -31412932 &  217937464 &   12832453\\
    37583858 &   72586497 &  -31412932 &    8890997 &  -61635435 &   -3624388\\
  -260766405 & -503453484 &  217937464 &  -61635435 &  427522394 &   25164249\\
   -15356594 &  -29630203 &   12832453 &   -3624388 &   25164249 &    1484928
\end{array}
\right)
\end{equation*}
\tiny
\begin{eqnarray*}
A^{3+}& := & \frac{1}{10^{10}}
\left(
\begin{array}{*{7}{r}}
   9911130076  &  9648505978 &  16664349190 &  9328739972 &  936324078 &  2061273472\\
   9648505978  &  24617473150 &  7220730652 &  10720412510 &  5131566980 & -23771485619\\
   16664349190 &  7220730652 &  46274936407 &  19324364441 &  1907848702 &  24288819932\\
   9328739972  &  10720412510 &  19324364441 &  15648801231 &  6256538527 & -1977377203\\
   936324078   &  5131566980 &  1907848702 &  6256538527 &  8925400310 & -12787276407\\
   2061273472  & -23771485619 &  24288819932 & -1977377203 & -12787276407 &  53782517217 & \ldots \\
  -18294412578 & -17850876485 & -45321487336 & -33168309077 & -15412214883 & -868355423\\
  -601754032   & -4390712397 &  8672825757 &  8687065569 & -943394603 &  13124014881\\
  -14953562741 & -9564752218 & -39851589035 & -20084517669 & -1900335064 & -18100206208\\
   7093611190  &  11908193393 & -121391261 &  6898543540 &  9339492106 & -20811949215\\
   2572438455  &  332196382 &  13726291191 &  369165438 & -4547421833 &  13307983466\\
  -16848495399 & -15745854235 & -39096523027 & -16479279662 & -1971257917 & -10251510005
\end{array}
\right. \\
& & \left.
\begin{array}{*{7}{r}}
 & -18294412578 & -601754032 & -14953562741 &  7093611190 &  2572438455 & -16848495399\\
 & -17850876485 & -4390712397 & -9564752218 &  11908193393 &  332196382 & -15745854235\\
 & -45321487336 &  8672825757 & -39851589035 & -121391261 &  13726291191 & -39096523027\\
 & -33168309077 &  8687065569 & -20084517669 &  6898543540 &  369165438 & -16479279662\\
 & -15412214883 & -943394603 & -1900335064 &  9339492106 & -4547421833 & -1971257917\\
\ldots & -868355423 &  13124014881 & -18100206208 & -20811949215 &  13307983466 & -10251510005\\
 &  74259569245 & -18500163943 &  44501159648 & -12830955695 & -2623240855 &  36726257026\\
 & -18500163943 &  26851070841 & -14712006368 & -14887887504 &  2474708345 &  145079825\\
 &  44501159648 & -14712006368 &  37588543382 &  2251374141 & -11236080891 &  33285652211\\
 & -12830955695 & -14887887504 &  2251374141 &  27039750365 & -11338452479 & -3895727030\\
 & -2623240855 &  2474708345 & -11236080891 & -11338452479 &  11811598967 & -13441522231\\
 &  36726257026 &  145079825 &  33285652211 & -3895727030 & -13441522231 &  38903952142
\end{array}\right)
\end{eqnarray*}
\normalsize
\begin{eqnarray*}
A^{4+}& := & \frac{1}{10^{10}}
\left(
\begin{array}{*{7}{r}}
   46800933 &  104186657 &  371956795 & -340593608&\\
   104186657 &  1038914521 &  1914699035 & -785600961&\\
   371956795 &  1914699035 &  4419934269 & -2744302275&\\
  -340593608 & -785600961 & -2744302275 &  2480213000&\ldots\\
  -137441315 & -720439710 & -1650449523 &  1014306440&\\
   85811179 & -1047011245 & -985196472 & -582494082&\\
  -55302691 & -230276564 & -583791454 &  406095485&\\
  -103112540 & -979995373 & -1830092885 &  775885584&
\end{array}
\right. \\
& & \left.
\begin{array}{*{7}{r}}
 & -137441315 &  85811179 & -55302691 & -103112540\\
 & -720439710 & -1047011245 & -230276564 & -979995373\\
 & -1650449523 & -985196472 & -583791454 & -1830092885\\
 &  1014306440 & -582494082 &  406095485 &  775885584\\
 &  616737186 &  383571355 &  217523075 &  688175720\\
 &  383571355 &  2057167346 &  62880166 &  962384782\\
 &  217523075 &  62880166 &  79623311 &  221456205\\
 &  688175720 &  962384782 &  221456205 &  925107090
\end{array}\right)
\end{eqnarray*}
\footnotesize
\begin{eqnarray*}
A^{5+}& := & \frac{1}{10^{10}}
\left(
\begin{array}{*{6}{r}}
     357787678 &  2667561490 & -3158511012 & -1321494554 & -719607624&\\
    2667561490 &  23782098207 & -27470938499 & -4462773752 & -7685990746&\\
   -3158511012 & -27470938499 &  31843561651 &  6263699767 &  8686163322&\\
   -1321494554 & -4462773752 &  6263699767 &  12586569384 & -595804570&\\
   -0719607624 & -7685990746 &  8686163322 & -595804570 &  2853679962& \ldots \\
   -0749034428 & -7099675831 &  8132219053 &  602235047 &  2416214541&\\
   -0746618224 & -7127076574 &  8165113979 &  602235047 &  2416214542&\\
    2667573049 &  23752965113 & -27444445656 & -4462773755 & -7685990744&\\
   -3159785490 & -27444445649 &  31809130251 &  6263699777 &  8686163317&\\
    356118081 &  2667573049 & -3159785490 & -1321494554 & -719607624&
\end{array}
\right. \\
& & \left.
\begin{array}{*{6}{r}}
 &  -749034428 & -746618224&   2667573049&  -3159785490&   356118081\\
 &  -7099675831 & -7127076574&   23752965113&  -27444445649&   2667573049\\
 &   8132219053 &  8165113979&  -27444445656&   31809130251&  -3159785490\\
 &   602235047 &  602235047&  -4462773755&   6263699777&  -1321494554\\
 \ldots &   2416214541 &  2416214542&  -7685990744&   8686163317&  -719607624\\
 &   2201093330 &  2167739656&  -7127076574&   8165113978&  -746618224\\
 &   2167739656 &  2201093330&  -7099675831&   8132219049&  -749034428\\
 &  -7127076574 & -7099675831&   23782098207&  -27470938496&   2667561490\\
 &   8165113978 &  8132219049&  -27470938496&   31843561644&  -3158511012\\
 &  -746618224 & -749034428&   2667561490&  -3158511012&   357787678 \end{array}\right)
\end{eqnarray*}
\normalsize
\begin{eqnarray*}
A^{1-}& := & \frac{1}{10^{10}}
\left(
\begin{array}{*{5}{r}}
   14307490741 &  2586781945 &  9478561500 & -8503040278&\\
   2586781945 &  24279644543 & -19491862247 &  21589760706&\\
   9478561500 & -19491862247 &  28964735982 & -22441455520&\\
  -8503040278 &  21589760706 & -22441455520 &  31332129348& \ldots \\
   13318186182 &  23618557200 & -13855425214 &  8889128831&\\
   8372793763 &  1525766074 &  5548458322 & -4968118814&\\
   1755357971 &  16315840647 & -13131926293 &  14676776308&
\end{array}
\right. \\
& & \left.
\begin{array}{*{4}{r}}
 &  13318186182 &  8372793763 &  1755357971\\
 &  23618557200 &  1525766074 &  16315840647\\
 & -13855425214 &  5548458322 & -13131926293\\
\ldots &  8889128831 & -4968118814 &  14676776308\\
 &  35089388707 &  7808231324 &  15932298681\\
 &  7808231324 &  4915982899 &  1034856802\\
 &  15932298681 &  1034856802 &  13507587054
\end{array}\right)
\end{eqnarray*}
\tiny
\begin{equation*}
A^{2-}:=\frac{1}{10^{10}}
\left(
\begin{array}{*{6}{r}}
   39301474130 & -12586488688 & -26325499489 &  4835146333 & -3200582867 &  39402624169\\
  -12586488688 &  4648340313 &  8427369658 & -1582302739 &  1181917652 & -12615039089\\
  -26325499489 &  8427369658 &  17666210517 & -3224681746 &  2142974291 & -26409599889\\
   4835146333 & -1582302739 & -3224681746 &  2763326005 & -402442928 &  4847377916\\
  -3200582867 &  1181917652 &  2142974291 & -402442928 &  300536605 & -3207844812\\
   39402624169 & -12615039089 & -26409599889 &  4847377916 & -3207844812 &  39528119078
\end{array}
\right)
\end{equation*}
\normalsize
\begin{equation*}
A^{3-}:=\frac{1}{10^{10}}
\left(
\begin{array}{*{4}{r}}
   73874538950 &  72373625861 &  10093264001 &  24183549820\\
   72373625861 &  93054314984 &  11635620671 &  29507918873\\
   10093264001 &  11635620671 &  1546846146 &  4342975799\\
   24183549820 &  29507918873 &  4342975799 &  59399032160
\end{array}
\right)
\end{equation*}
\tiny
\begin{equation*}
A^{5-}:=\frac{1}{10^{10}}
\left(
\begin{array}{*{6}{r}}
    9930361952 & - 8840932504 &   3083822328 &  19041777196 &            0 &  8749601119  \\
  - 8840932504 &  52724642820 &  25782427487 & -3432529692 &  8749601126 &  7  \\
    3083822328 &  25782427487 &  48331024300 &            0 &  19041777196 &  3432529692  \\
   19041777196 & - 3432529692 &            0 &  48331024300 & -3083822328 &  25782427477  \\
             0 &   8749601126 &  19041777196 & -3083822328 &  9930361952 &  8840932508  \\
    8749601119 &            7 &   3432529692 &  25782427477 &  8840932508 &  52724642813
\end{array}
\right)
\end{equation*}
\section{Appendix: Table for the calculus}\label{app:calctable}
\normalsize
For an easier reading we write
\begin{eqnarray*}
L & := & \parbox[c]{1cm}{
\begin{tikzpicture}
\filldraw[fill=black] (1,1) circle (0.07cm);
\filldraw[fill=black] (0,1) circle (0.07cm);
\filldraw[fill=black] (1,0) circle (0.07cm);
\filldraw[fill=black] (0,0) circle (0.07cm);
\draw (0,0)--(0,1)--(1,0)--(1,1)--(0,1);
\draw (1,1)--(0,0)--(1,0);
\end{tikzpicture}}
\;+\;
\parbox[c]{1cm}{
\begin{tikzpicture}
\filldraw[fill=black] (1,1) circle (0.07cm);
\filldraw[fill=black] (0,1) circle (0.07cm);
\filldraw[fill=black] (1,0) circle (0.07cm);
\filldraw[fill=black] (0,0) circle (0.07cm);\end{tikzpicture}}
\;-\frac{1}{34.7858}, \\
R & := & \sum\limits_{i=0}^5\left\llbracket\left(\underline{g_i}^+\right)^TA^{i+}\underline{g_i}^+\right\rrbracket_{\sigma_i}+\sum\limits_{i\in\left\{1,2,3,5\right\}}\left\llbracket\left(\underline{g_i}^-\right)^TA^{i-}\underline{g_i}^-\right\rrbracket_{\sigma_i} \\
 & & + \sum\limits_{i=0}^4\left\llbracket\left(\overline{\underline{g_i}^+}\right)^TA^{i+}\overline{\underline{g_i}^+}\right\rrbracket_{\sigma_i}+\sum\limits_{i=1}^3\left\llbracket\left(\overline{\underline{g_i}^-}\right)^TA^{i-}\overline{\underline{g_i}^-}\right\rrbracket_{\sigma_i}.
\end{eqnarray*}
In the following table the $(i,j)$'th element of the matrix $A^{k*}$ is denoted by $a_{i,j}^{k*}$. Because of the symmetry, we don't have to distinguish between $a_{i,j}^{k*}$ and $a_{j,i}^{k*}$. Each row in the following table belongs to the coefficient of $G_i$ in equation \eqref{eq:main}.
\tiny
\begin{longtable}{c|c|c|c|p{8cm}}
$i$ & $6!L$ & $6!R$ & $\left(6!L-6!R\right)10^3$ & $6!R$ (in variables) \\ \hline \hline \endhead
$  0$ &$  699.3019$ &$  699.3011$ & $    0.8277$ & $414720a^{0+}_{0,0}$\\ \hline
$  1$ &$  411.3019$ &$  411.3012$ & $    0.7348$ & $27648a^{0+}_{0,0}+55296a^{0+}_{0,1}+768a^{1+}_{0,0}$\\ \hline
$  2$ &$  267.3019$ &$  267.3018$ & $    0.0632$ & $13824a^{0+}_{0,1}+1728a^{0+}_{1,1}+13824a^{0+}_{0,2}+384a^{1+}_{0,1}+48a^{3+}_{0,0}$\\ \hline
$  3$ &$  123.3019$ &$  123.3018$ & $    0.0688$ & $5184a^{0+}_{1,1}+1152a^{1+}_{0,3}+1296a^{4+}_{7,7}$\\ \hline
$  4$ &$   75.3019$ &$   75.3019$ & $    0.0242$ & $2304a^{0+}_{1,2}+32a^{1+}_{1,1}+128a^{1+}_{0,3}+128a^{1+}_{1,3}+128a^{1+}_{0,6}+32a^{3+}_{0,6}+16a^{3+}_{11,11}+288a^{4+}_{5,7}+8a^{1-}_{0,0}$\\ \hline
$  5$ &$   27.3019$ &$   27.3019$ & $    0.0240$ & $768a^{0+}_{2,2}+256a^{1+}_{1,3}+128a^{1+}_{3,3}+128a^{1+}_{8,8}+64a^{3+}_{0,8}+576a^{4+}_{3,7}$\\ \hline
$  6$ &$   27.3019$ &$   27.3019$ & $    0.0395$ & $27648a^{0+}_{4,4}+768a^{1+}_{3,3}+3456a^{4+}_{1,7}$\\ \hline
$  7$ &$   27.3019$ &$   27.3019$ & $    0.0274$ & $6912a^{0+}_{3,4}+192a^{1+}_{3,6}+12a^{3+}_{6,6}+48a^{3+}_{1,11}+432a^{4+}_{1,5}+432a^{4+}_{0,7}+432a^{4+}_{1,7}+12a^{3-}_{1,1}$\\
\hline

\end{longtable}

\end{appendix}
\end{document}